\documentclass[a4paper, 10pt]{amsart}

\usepackage{amsmath, amscd}
\usepackage{amssymb, color}
\usepackage{stmaryrd}
\usepackage{enumitem}
\usepackage{tikz-cd}

\definecolor{green-url}{RGB}{0,128,0}
\definecolor{blue-url}{RGB}{0,0,205}
\definecolor{red-url}{RGB}{227,0,34}
\definecolor{ashgrey}{rgb}{0.7, 0.75, 0.71}
\usepackage[colorlinks=true,pdftex,unicode=true,linktocpage,bookmarksopen,hypertexnames=false,citecolor=green-url, linkcolor=blue-url, urlcolor=red-url]{hyperref}

\newcommand{\defit}[1]{{\textsf{#1}}}

\renewcommand{\epsilon}{\varepsilon}
\renewcommand{\phi}{\varphi}

\theoremstyle{plain}
\newtheorem{theorem}{Theorem}[section]
\newtheorem{proposition}[theorem]{Proposition}
\newtheorem{lemma}[theorem]{Lemma}
\newtheorem{corollary}[theorem]{Corollary}
\newtheorem{problem}[theorem]{Problem}

\theoremstyle{definition}
\newtheorem{example}[theorem]{Example}

\newtheorem{definition}[theorem]{Definition}

\theoremstyle{remark}
\newtheorem{remark}[theorem]{Remark}
\newtheorem{remarks}[theorem]{Remarks}

\newcommand{\N}{\mathbb N}
\newcommand{\Z}{\mathbb Z}

\newcommand{\vv}{\mathsf{v}}
\DeclareMathOperator{\spec}{spec}

\newcommand{\DP}{\negthinspace : \negthinspace}

\setlist[enumerate,1]{itemsep=0.05cm, label=\textup{(}\normalfont\arabic*\textup{)}}
\setlist[enumerate,2]{itemsep=0.05cm, label=\textup{(}\normalfont\roman*\textup{)}}
\setlist[itemize,1]{itemsep=0.05cm}
\numberwithin{equation}{section}

\makeatletter
\@namedef{subjclassname@2020}{%
	\textup{2020} Mathematics Subject Classification}
\makeatother

\title[On Transfer Homomorphisms in Commutative Rings with Zero-divisors]{On Transfer Homomorphisms \\ in Commutative Rings with Zero-divisors}

\author{Aqsa Bashir}
\email{aqsa.bashir@uni-graz.at}
\address{University of Graz, Department of Mathematics and Scientific Computing, NAWI Graz, Heinrichstrasse 36, 8010 Graz, Austria }
\author{Mara Pompili}
\email{mara.pompili@uni-graz.at}
\address{University of Graz, Department of Mathematics and Scientific Computing, NAWI Graz, Heinrichstrasse 36, 8010 Graz, Austria }

\thanks{This research was funded in whole, or in part, by the Austrian Science Fund (FWF) [10.55776/P36852]. For the purpose of open access, the author has applied a CC BY public copyright license to any Author Accepted Manuscript version arising from this submission }

\begin{document}

\begin{abstract}
    We study the arithmetic of monoids of regular elements  of commutative rings with zero-divisors. Our focus is on Krull rings and on some of their generalizations (such as weakly Krull rings and C-rings). We establish sufficient conditions for a subring $R$ of a Krull ring $D$ guaranteeing that the inclusion $R^{\bullet} \hookrightarrow D^{\bullet}$ of the respective monoids of regular elements is a transfer homomorphism. The arithmetic of the Krull monoid $D^{\bullet}$ is well studied and the existence of a transfer homomorphism implies that $R^{\bullet}$ and $D^{\bullet}$ share many arithmetic properties.
\end{abstract}

\maketitle


\section{Introduction} \label{1}
For a long time factorization theory of commutative rings had its focus on integral domains. Arithmetic studies in commutative rings with zero-divisors started much later and there are two main approaches. One direction, initiated by Daniel D. Anderson \cite{AM85, AVL96}, considers the full ring and deals with factorizations of zero-divisors. A large variety of results has been achieved so far (for a sample, see \cite{AJ17,JM22, JMN21}), but the arithmetic properties under consideration are less fine than those studied for integral domains. 

The other direction (for a sample, see \cite{CO22, GRR15, HZ25}) restricts its attention to the monoid of regular elements of the ring. Such a monoid is a commutative cancellative semigroup with identity. There is a well-developed factorization theory for this class of monoids  and, inside this class, Krull monoids are most investigated. Their arithmetic can be studied, among others,  with methods from additive combinatorics and there is an abundance of results on their arithmetic (we can only refer to some recent results and surveys, \cite{Gri22, Sch16}). A crucial strategy in factorization theory makes use of transfer homomorphisms. In order to study a monoid $H$ (of a given class of monoids) one constructs a transfer homomorphism $\theta \colon H \to B$, where $B$ is simpler (in some aspects), and $\theta$ pulls back the arithmetic results achieved for $B$ to the original monoid of interest $H$. The classic example of a transfer homomorphism is the epimorphism from a Krull monoid to a monoid of zero-sum sequences over the class group of the Krull monoid. 

In the last decade, a pivotal result by Smertnig \cite{Sme13} initiated investigations which non-Krull monoids and domains allow a transfer homomorphism to a Krull monoid and subsequently to a monoid of zero-sum sequences (for an overview, we refer to \cite{BR22, GZ20} and, for some recent progress,  to \cite{Rag25}). The present paper is in the vein of these investigations. We start with a general transfer result (Theorem \ref{prop2}). For a given commutative ring $D$, we establish sufficient conditions for a subring $R \subseteq D$ implying that the inclusion $R^{\bullet} \hookrightarrow D^{\bullet}$ of the respective monoids of regular elements is a transfer homomorphism. We then apply this result (among others) to the case when $D$ is a Krull ring (Theorem \ref{thm: transferkrull}). If $D$ is a Krull ring, then $D^{\bullet}$ is a Krull monoid and the existence of a transfer homomorphism $R^{\bullet} \hookrightarrow D^{\bullet}$ implies that all the well investigated properties of the Krull monoid $D^{\bullet}$ hold true also for the monoid of regular elements of the subring $R \subseteq D$.

In Section \ref{2}, we gather the required background on the arithmetic of monoids, on the ideal theory of commutative rings with zero-divisors, and on transfer homomorphisms. In Section \ref{3}, we establish a result on the existence of a transfer homomorphism in a general setting, which we then, in Section \ref{4},  apply to Krull rings and C-rings. In Section \ref{5} we study the arithmetic of weakly Krull rings and weakly Krull monoids in the setting of $t$-Marot rings.
\section{Prerequisites} \label{2}
\smallskip

\subsection{Monoids} By a monoid, we mean a commutative, cancellative semigroup with identity. Let $H$ be a monoid. We denote by $\mathbf{q}(H)$ the quotient group of $H$, and by $H^\times$ the group of units. The \defit{reduced monoid} is the monoid given by $\{aH^\times\mid a\in H\}$. We say that two elements $a,b\in H$ are \defit{associated}  if $a=ub$ for some $u\in H^\times$, and we write $a\simeq b$. We say that $a$ divides $b$ in $H$ if $bH\subseteq aH,$ and we write $a|_Hb$. An $s$-\defit{ideal} of $H$ is a subset $\mathfrak{a}\subseteq H$ such that $\mathfrak{a}H=\mathfrak{a}.$ We say that an s-ideal $\mathfrak{a}$ is a \defit{prime s-ideal} if for every $a,b\in H$ with $ab\in \mathfrak{a}$ we have $a\in \mathfrak{a}$ or $b\in \mathfrak{a}.$ We will denote by $s$-$\spec(H)$ the set of prime s-ideals of $H$. For sets $X,Y\in \mathbf{q}(H)$, set $(X\DP Y)=\{a\in \mathbf{q}(H)\mid aY\subseteq X\},$ and $X_{v_H}:=X_v:=(H\DP (H\DP X)).$  We say that $X$ is a \defit{fractional} $v$-\defit{ideal} of $H$ if $cX\subseteq H$ for some $c\in H$ and $X_v=X$, and we say that $X$ is a $v$-\defit{ideal} if $X\subseteq H$ and $X_v=X.$ We denote by $v$-$\spec(H)$ the set of prime $v$-ideals of $H$, by $\mathfrak{X}(H)$ the set of minimal prime $v$-ideals, by $(\mathcal{F}_v(H),\cdot_v)$ the semigroup of fractional $v$-ideals of $H$ with the $v$-multiplication, and by $\mathcal{I}_v(H)$ the subsemigroup of $v$-ideals of $H$. Finally, we say that the monoid $H$ is \defit{completely integrally closed} if $H$ coincides with its \defit{complete integral closure} $\widehat{H}:=\{x\in\mathbf{q}(H)\mid \text{ there exists } c\in H : cx^n\in H, \text{ for all } n\in \mathbb{N}\}$.

\subsection{Rings} By a ring, we mean a non-zero commutative ring with unit element. Let $R$ be a ring with total quotient ring $\mathsf T(R)$. We denote by $R^{\times}$ the group of units of $R$, by $\mathsf Z(R)$ the set of zero-divisors (note that $0 \in \mathsf Z(R)$), by $R^{\bullet} =R \setminus \mathsf Z(R)$ the set of regular elements of $R$.
A subset $X \subseteq \mathsf T(R)$ is said to be \defit{regular} if $X^{\bullet}:=X \setminus \mathsf Z(\mathsf T(R)) \neq \emptyset$. Thus an ideal $I$ in $R$ is called a \defit{regular ideal} if it contains a regular element. Note that $R^{\bullet}$ is a monoid and $\mathsf T(R)^{\times}= \mathbf{q}(R^{\bullet})=\mathsf T(R)^{\bullet}$.

We denote by $\spec(R)$, resp. $\spec_r(R)$ the set of all prime ideals, resp. the set of all regular prime ideals. Moreover, we denote by $\max(R)$, resp. $\max_r(R)$ the set of all maximal ideals, resp. the set of all regular maximal ideals.  
Let $P$ be a regular prime ideal of $R$, and let $\mathrm{ht}P$ denote the \defit{height} of $P$ and $\dim (R)$ denote the dimension of $R$. The \defit{regular height} of $P$ is defined by $\mathrm{reg}$-$\mathrm{ht} P=\sup \{n\mid P_1\subsetneq\cdots\subsetneq P_n=P,\, P_i\in \spec_r(R) \text{ for all $i \in [1, n]$}\}.$ Then the \defit{regular dimension} of $R$ is defined by \[\mathrm{reg}\text{-}\dim(R)=\sup\{\mathrm{reg}\text{-}\mathrm{ht} P\mid P\in \spec_r(R)\}. \] Thus $\mathrm{reg}$-$\mathrm{ht} P \le \mathrm{ht}P$, $\mathrm{reg}$-$\dim(R)\le \dim (R)$. Denote by $\mathfrak{X}_r(R)$ the set of all regular prime ideals $P\in \spec(R)$ such that $\mathrm{reg}$-$\mathrm{ht}P=1$.

Let $S$ be a multiplicatively closed set of $R$. We consider the following localization of $R$ with respect to $S,$ \begin{itemize}
    \item $R_S:=\{\frac{a}{b}\mid a\in R,\text{ and } s\in S \},$
    \item $R_{(S)}:=\{\frac{a}{s}\mid a\in R, \, \text{and} \, s\in S^\bullet\},$
    \item $R_{[S]}:=\{z\in \mathsf T(R)\mid zs\in R\, \text{for some} \, s\in S\}.$ 
\end{itemize}
Clearly, $R_{(S)}\subseteq R_{[S]}\subseteq \mathsf T(R)$, and if $S\subseteq R^\bullet,$ then $R_{(S)}= R_{S}$. If $P\in \spec (R),$ then we set $R_{(P)}:=R_{(R\setminus P)}$ and $R_{[P]}:=R_{[R\setminus P]}.$ Moreover, if $I$ is an ideal of $R$, then $[I]R_{[P]}:=\{x\in \mathsf T(R)\mid xa\in I\, \text{for some} \, a\in R\setminus P\}$ is an ideal of $R_{[P]},$ and $[P]R_{[P]}$ is a prime ideal of $R_{[P]}.$

\subsection{Atoms in monoids and rings} Let $H$ be a monoid and let $a$ be a non-unit of $H$. Then $a$ is called an \defit{atom} of $H$ if $a=uv$ for some $u, v \in H$ implies that $u \in H^{\times}$ or $v \in H^{\times}$. Moreover, $a$ is said to be a \defit{prime element} of $H$ if $a \mid uv$ for some $u,v \in H$ implies $a \mid u$ or $a\mid v$. We say that $H$ is \defit{atomic} if every non-unit can be written as a finite product of atoms. For each non-unit $a\in H$, we let 
\[\mathsf{L}_H(a)=\mathsf{L}(a)=\{k\in\mathbb{N}\mid a \text{ is a product of $k$ atoms of $H$}\}\subseteq\mathbb{N}\]  
be the \defit{set of lengths} of $a$. Furthermore, we set $\mathsf{L}_H(a)=\mathsf{L}(a)=\{0\}$ for each $a\in H^{\times}$. An atomic monoid $H$ is said to be \defit{factorial} if every atom of $H$ is a prime element and it is called \defit{half-factorial} if $|\mathsf{L}(a)|=1$ for all $a\in H$. Observe that every factorial monoid is half-factorial. Moreover, we say that $H$ is a \defit{BF-monoid} if $H$ is atomic and all sets of lengths are finite, and that $H$ is \defit{length-factorial} if each two distinct factorizations of any element have distinct factorization lengths.

\smallskip

 Let $R$ be a ring. For the purpose of this paper, we will study factorization properties of the ring $R$ focusing solely on the study of factorization properties of its regular elements. Indeed, we say that a regular element $a\in R^\bullet$ is an \defit{atom} if $a = bc$ for some $b, c \in R$ implies that $b \in R^{\times}$ or $c \in R^{\times}$. We call the ring $R$ \defit{atomic} (respectively, \defit{half-factorial}, \defit{BF-ring}, \defit{length-factorial}) if $R^{\bullet}$ is an atomic monoid (respectively, a half-factorial monoid, a BF-monoid, a length-factorial monoid). The set of lengths of factorizations in $R$ is denoted by $\mathcal{L}(R)$ and it is the set $\{\mathsf L_{R^\bullet}(a)\mid a\in R^\bullet\}.$
We point out that factorization lengths and related arithmetical invariants can be developed also for non-cancellative monoid, and so in presence of non-regular elements, as discussed for instance in \cite{AVL96, Cos25, CT23, Tri22}.

\subsection{Ideal Theory of Rings}\label{ssec:ideal theory} Let $R$ be a ring, and $T=\mathsf{T}(R)$ its total quotient ring. An element $x\in \mathsf T(R)$ is called \defit{almost integral} over $R$ if there exists some $c\in R^\bullet$ such that $cx^n\in R$ for all $n\in \N.$ We call \[\widehat{R}=\{x\in \mathsf{T}(R)\mid x\, \text{is almost integral over} \, R\}\] the \defit{complete integral closure} of $R$. Moreover, we say that $R$ is \defit{completely integrally closed} if $R=\widehat{R}.$  
   
    For all $X,Y\subseteq T$, we set 
    \[(X\DP Y) :=\{a\in T\mid aY\subseteq X\} \] 
 An $R$-submodule of $\mathsf{T}(R)$ is called a \defit{Kaplansky fractional ideal} of $R$. We denote by $\mathsf{K}(R)$ the set of Kaplansky fractional ideals. An \defit{(integral) ideal} of $R$ is a Kaplansky fractional ideal of $R$ that is contained in $R$.
For every  $I \in \mathsf{K}(R)$, denote by $I^{-1}:=(R\DP I) =\{x \in \mathsf T(R) \mid xI \subseteq R\}$. Note that $I^{-1} \in \mathsf K(R)$. Moreover, we have that $I_{v_R}:=I_v:=(I^{-1})^{-1}$ and $I_{t_R}:=I_t:=\cup \{J_v \mid J\in \mathsf{K}(R), J\subseteq I, J \,\,\text{finitely generated}\}$ are in $\mathsf K(R)$.

\begin{definition} Let $\ast= v \text{ or } t$. Define $\ast \colon \mathsf{K}(R)\to\mathsf{K}(R)$ that maps  $I \in \mathsf{K}(R)$ into $I_{\ast}\in \mathsf{K}(R)$. The operation $\ast$ is called the $\ast$-\textsf{operation} on $R$. An element $I$ of $\mathsf{K}(R)$ is said to be a $\ast$-ideal if $I_{\ast}=I$. Moreover, we will often call a $v$-ideal a \defit{divisorial} ideal.
\end{definition}

\begin{proposition}[{\cite[Proposition 2.4.26]{El19}}]\label{prop:semistar operation}
    Let $R$ be a ring, $I,J\in \mathsf{K}(R)$ and $*=v$ or $t$.
    \begin{enumerate}[label=\textup(\normalfont \arabic*\textup)]
        \item If $I\subseteq J$, then $I_*\subseteq J_*$.
        \item $I\subseteq I_*=(I_*)_*$.
        \item $(IJ)_*=(I_*J)_*=(I_*J_*)_*.$
        \item $aI_*\subseteq (aI)_*$ for all $a\in T$, and $aI_*= (aI)_*$ for all $a\in T^\bullet.$ 
        \item $((I^{-1})^{-1})^{-1}=I^{-1}.$
         \item $R_*=R$.
    \end{enumerate} 
\end{proposition}

Formally speaking, Parts (1)--(4) in the previous proposition state that the $v$-operation and $t$-operation are semi-star operations. For more details on semi-star operations, see for instance \cite{El19}. Note that Part (5) implies that $I^{-1}$ is a divisorial ideal for every $I\in \mathsf K(R).$ 
We say that $X\in \mathsf{K}(R)$ is a \defit{regular fractional $v$-ideal} if $X$ contains a regular element, $cX\subseteq R$ for some $c\in R^\bullet$, and is divisorial, and that $X$ is a \defit{regular $v$-ideal} if $X\subseteq R$, is regular, and is divisorial. We denote by $(\mathcal{F}_v(R), \cdot_v)$ the semigroup of regular fractional $v$-ideals of $R$, and by $\mathcal I_v(R)$ the subsemigroup of regular $v$-ideals of $R$. Then $\mathcal I^{\ast} _v(R) = \mathcal I_v(R) \cap \mathcal F_v(R)^{\times}$ is the monoid of $v$-invertible regular $v$-ideals of $R$ and its quotient group is $\mathcal F_v(R)^{\times}$. The $v$-\defit{class group} $\mathcal{C}_v(R)$ of $R$ is then defined as the quotient $\mathcal{F}_v(R)^\times/\mathcal{H}(R)$, where $\mathcal{H}(R)=\{aR\mid a\in \mathsf{T}(R)^\bullet\},$ and similarly the $t$-\defit{class group}  $\mathcal{C}_t(R)$ of $R$ is $\mathcal{F}_t(R)^\times/\mathcal{H}(R)$, where $\mathcal{F}_t(R)$ is the set of regular fractional $t$-ideals of $R$.
We denote by $v$-$\spec(R)$, resp. $v$-$\spec_r(R)$, the set of all prime $v$-ideals, resp. the set of all regular prime $v$-ideals. Moreover, we denote by $v$-$\max(R)$, resp. $v$-$\max_r(R)$, the set of all maximal $v$-ideals, resp. the set of all regular maximal $v$-ideals. Finally we say that $R$ is a \defit{Mori} ring if $R$ satisfies the ascending chain conditions on regular divisorial ideals.

\begin{definition} A ring $R$ is said to be a
\begin{enumerate}
    \item \defit{Marot ring} if each regular ideal of $R$ is generated by its regular elements, that is $I=I^\bullet R$ for every ideal $I$ of $R$;
    \item  \defit{$t$-Marot} if each regular fractional $t$-ideal of $R$ is $t$-generated by its regular elements,  that is $I=(I^\bullet R)_t$ for every $t$-ideal $I$ of $R$;
    \item  \defit{$v$-Marot ring} if each regular fractional $v$-ideal of $R$ is $v$-generated by its regular elements,  that is $I=(I^\bullet R)_v$ for every $v$-ideal $I$ of $R$;
\end{enumerate}
\end{definition}
If $R$ is noetherian, then $R$ is a Marot ring, and if $R$ is a Marot ring, then $R$ is a $v$-Marot ring, see \cite[Proposition 3.3]{GRR15} for a proof. Moreover, each $v$-ideal is a $t$-ideal thus a $t$-Marot ring is a $v$-Marot ring. Also, if $R$ is a Mori ring, then each $v$-ideal is a $t$-ideal.

\begin{proposition}[{\cite[Theorem 3.5]{GRR15}}]\label{class group thm}
    Let $R$ be a $v$-Marot ring. \begin{enumerate}
        \item If $\emptyset\ne\mathfrak{a}\subseteq \mathbf{q}(R^\bullet)$ is such that $c\mathfrak{a}\subseteq R^\bullet$ for some $c\in \mathbf{q}(R^\bullet)$, then $(\mathfrak a_{v_R})^\bullet=\mathfrak{a}_{v_{R^\bullet}}.$
        \item The map \[\iota\colon \mathcal{F}_{v_R}(R)\to \mathcal{F}_{v_{R^{\bullet}}}(R^\bullet)\setminus\{\emptyset\},\qquad I\mapsto I^\bullet\] is an inclusing preserving semigroup isomorphism. In particular, $\mathcal{I}^*_{v_R}(R)\cong \mathcal{I}_{v_{R^{\bullet}}}^*(R^\bullet)$ and there is an inclusion preserving bijection between ${v_R}$-$\spec_r(R)$ and ${v_{R^{\bullet}}}$-$\spec(R^\bullet)\setminus\{\emptyset\}.$
    \end{enumerate}
\end{proposition}
An equivalent statement for $t$-ideals is provided in \cite[Theorem 2.7]{CO22}.
\subsection{Transfer homomorphisms}
A monoid homomorphism $\theta : H \to B$ is called a \defit{transfer homomorphism} if it satisfies the following properties: 
\begin{itemize}
	\item[\bf{(T1)}] $B=\theta(H)B^{\times}$ and $\theta^{-1}(B^{\times})=H^{\times}$.
	\item[\bf{(T2)}] If $u \in  H$, $b, c \in B$ and $\theta(u)=bc$, then there exist $v,w \in H$ such that $u=vw$, $\theta(v)\simeq b$ and $\theta(w)\simeq c$.
\end{itemize}
 A transfer homomorphism is significant because it allows to pull back the arithmetic properties from the target object to the source object (see the next Proposition).

\begin{proposition}[{\cite[Proposition 3.2.3]{GHK06}}]\label{proptransfhomo}
    Let $\theta\colon H\to B$ be a transfer homomorphism. Then \begin{enumerate}[label=\textup{(}\normalfont \arabic*\textup{)}]
        \item An element $u\in H$ is an atom of $H$ if and only if $\theta(u)$ is an atom of $B$.
        \item $H$ is atomic if and only if $B$ is atomic.
        \item If $H$ is atomic, then \[\mathcal{L}(H)=\{\mathsf L_H(a)\mid a\in H\}=\{\mathsf L_B(b)\mid b \in B\}=\mathcal{L}(B).\]
        \item $H$ is a BF-monoid if and only if $B$ is a BF-monoid.
    \end{enumerate}
\end{proposition}

\section{A General Transfer Result} \label{3}
In this Section we provide sufficient conditions on a ring $R\subseteq D$  that ensure the inclusion $R^\bullet\hookrightarrow D^\bullet$ is a transfer homomorphism (see Theorem \ref{prop2}). Before doing that, we need some preliminary lemmas.
	\begin{lemma}\label{lemma}
		Let $R \subseteq D$ be rings with $\mathsf T(R) =\mathsf T(D)$. 
		\begin{enumerate}[label=\normalfont(\arabic*)]	
			\item If $D^{\times} \cap R= R^{\times}$, then $D^\times \cap R^{\bullet}= R^\times$.
			\item If $D=RD^{\times}$, then $D^{\bullet}=R^{\bullet}D^\times$. 
			\item $D^{\bullet} \cap R = R^{\bullet}$ and $\mathsf Z(D) \cap R =\mathsf Z(R)$.
		\end{enumerate}
	\end{lemma}
	
	\begin{proof}	
		(1) Clearly, $D^\times \cap R^\bullet\subseteq D^\times \cap R=R^\times,$ and $R^\times \subset D^\times \cap R^\bullet.$\\
		(2) To prove $D^{\bullet}=R^{\bullet}D^\times$, first let $x \in D^{\bullet}$. If $x \in D^\times$, then $x=1\cdot x \in R^{\bullet} D^\times$. Otherwise, assume $x=ru$ where $r \in R$, $u \in D^{\times}$. We claim that $r \in R^{\bullet}$. To verify the claim, assume on contrary that $r \in \mathsf Z(R)$. By definition of a zero-divisor, there exists a $0 \neq t \in R$ such that $tr=0$, but then $tx=tru=0$ implying that $x \in \mathsf Z(D)$ which is a contradiction to our assumption that $x \in D^{\bullet}$. \\
		For the opposite inclusion, first note that $R^{\bullet}D^\times \subseteq D$. Now let $ x \in R^{\bullet}D^\times$. To prove that $ x \in D^{\bullet}$, we show that $x \notin \mathsf Z(D)$. Assume $x=r\epsilon$ for some $r \in R^{\bullet}$ and some $\epsilon \in D^\times$. Now assume on contrary that $x \in \mathsf Z(D)$. That means there exists some $0 \neq t \in D$ such that $tx=0$, consequently $tr\epsilon=0$ which implies $tr=0$ as $\epsilon$ is a unit in $D$. Therefore $r \in \mathsf Z(R)$ (indeed, as $0 \neq t \in D=RD^{\times}$ say $t=su$ for some $s \in R$ and $u\in D^{\times}$ then $s \neq 0$ which follows that there exists a non-zero element $s$ in $R$ such that $sr=0$ in $R$), which is a contradiction to the choice of $r$ and thus the assertion follows. \\
		(3) It is clear. \\
	\end{proof}
    \begin{remark}
		Let $R$ be a ring, $\mathsf{T}(R)$ its total quotient ring, and $\mathbf{q}(R^\bullet)$ the quotient group of the monoid $R^\bullet$. By definition,
		\[\mathsf{T}(R)=\left\{\frac{r}{s}\mid r\in R, s\in R^\bullet\right\}, \qquad\mathbf{q}(R^\bullet)=\left\{\frac{r}{s}\mid r,s\in R^\bullet\right\}.\]
		Clearly, $\mathbf{q}(R^\bullet)\subseteq \mathsf{T}(R)$ and $\mathbf q(R^{\bullet})=\mathsf T(R)^{\bullet} =\mathsf T(R)^{\times}$.

		It follows that if $R\subseteq D$ are rings with $\mathsf T(R) =\mathsf T(D)$, then $\mathbf{q}(R^{\bullet})=\mathsf T(R)^{\times} =\mathsf T(D)^{\times}=\mathbf{q}(D^{\bullet})$ and 
		\[ (R\DP D)=\{a\in \mathsf{T}(D)\mid aD\subseteq R\} \] and \[ (R^\bullet \DP D^\bullet)=\{a\in \mathbf{q}(D^\bullet)\mid aD^\bullet\subseteq R^\bullet\}. \]
	\end{remark}
\begin{lemma}\label{lem:totquotring}
    Let $R\subseteq D$ be a ring extension.
    \begin{enumerate}[label=\textup(\normalfont\arabic*\textup)]
        \item If $(R\DP D)$ contains a regular element of $D$, then $\mathsf{T}(R)=\mathsf{T}(D).$
        \item If $\mathsf T(R) =\mathsf T(D)$ and $D=RD^{\times}$, then $$(R\DP D)\cap \mathsf{T}(D)^\times=(R^\bullet \DP D^\bullet).$$
    \end{enumerate}
\end{lemma}
\begin{proof}
 (1) Let $x\in (R\DP D)$ be a regular element of $D$.
    If $z=a^{-1}b \in \mathsf{T}(D)$, where $a \in D^\bullet$, and $b\in D$, then $ax,bx \in (R\DP D) \subseteq  R$ and $z=(ax)^{-1}(bx)\in \mathsf{T}(R).$ 
    Conversely, observe that $R^\bullet\subseteq D^\bullet.$ Indeed, if $r\in R^\bullet,$ and $d\in D$ such that $rd=0,$ then $0=rdx$, and hence $d=0$, since $x\in (R\DP D)\cap D^\bullet.$ In particular, $\mathsf{T}(R)\subseteq \mathsf{T}(D).$

    (2) Let $a\in \mathsf{T}(D)^\times$ such that $aD\subseteq R$. In particular, $aD^\bullet\subseteq R^\bullet$. Indeed, let $d\in D^\bullet$ and $r=ad\in R$. Then $r$ is regular in $R$ since if there is an element $0 \neq t\in R$ such that $rt=0$, then $adt=0$, and that would imply that $dt=0\in D$, and $d\in \mathsf{Z}(D)$, whence $a\in (R^\bullet \DP D^\bullet)$.
		Conversely, let $a\in \mathbf{q}(D^\bullet)$ such that $aD^\bullet\subseteq R^\bullet$. Since $\mathbf q(D^{\bullet})=\mathsf T(D)^{\times}$ we get that $a\in \mathsf{T}(D)^\times,$ so we need to prove that $ad\in R$ for every $d\in D$. By our assumption $D=RD^\times$, 
		so every $d\in D$ can be written as $d=r\epsilon$ with $r\in R$, and $\epsilon\in D^\times.$ Then $ad\in R$ if and only if $ar\epsilon\in R$.
		Since every unit of $D$ is a regular element of $D$, we get that $a \epsilon \in a D^{\bullet} \subseteq R^{\bullet} \subseteq R$, hence $ar\epsilon \in R$, whence $a \in (R\DP D)$.
\end{proof}
	
 Next we present our general transfer result. Note that $(R\DP D)$ is divisorial, and thus if $(R\DP D)$ is a maximal ideal of $R$, then $(R\DP D)\in v$-$\max(R)$.

 \begin{theorem}\label{prop2}
		Let $D$ be a ring, $R \subseteq D$ a subring with $\mathsf T (R)=\mathsf T (D)$ such that
\[
D = RD^{\times}, \  D^{\times} \cap R = R^{\times}, \quad \text{and} \quad { (R \DP D) \in v\text{-}\max (R)}  \,.
\]
Then the inclusion $R^{\bullet} \hookrightarrow D^{\bullet}$ is a transfer homomorphism.  In particular, if $R$ is atomic, we have that $\mathcal{L}(R)=\mathcal{L}(D).$
	\end{theorem}
 \begin{proof}
		\textbf{(T1)} is satisfied by Lemma \ref{lemma}.\\
		To prove \textbf{(T2)}, let $u \in R^{\bullet}$ and $b,c \in D^{\bullet}$ such that $u=bc$. We must prove that there exist $v, w \in R^{\bullet}$ and $\epsilon, \eta \in D^\times$ such that $u=vw$, $v=b\epsilon$ and $w=c\eta$. By Lemma \ref{lemma}(2), $D^{\bullet}=R^{\bullet}D^\times$ so there exist $b_0, c_0 \in R^{\bullet}$ and $\epsilon_0, \eta_0 \in  D^\times$ such that $b=b_0 \epsilon_0$ and $c=c_0 \eta_0$. We distinguish two cases.\\
		CASE I: $b_0 \in \mathfrak{m}=(R\DP D)$.
		
		Then $b_0 x \in R$ for any $x \in D$, in particular $b_0 \epsilon_0 \eta_0 \in R$. Also, $b_0 \epsilon_0 \eta_0 \in R^{\bullet}D^\times=D^{\bullet}$ therefore Lemma \ref{lemma}(3) implies that $b_0 \epsilon_0 \eta_0 \in R^{\bullet}$. Now we set $v=b_0 \epsilon_0 \eta_0$, $w=c_0$, $\epsilon= \eta_0$, $\eta={\eta_0}^{-1}$ and we are done.  \\
	CASE II: $b_0\notin \mathfrak{m}=(R\DP D).$

 Then $\mathfrak{m}\subsetneq b_0R+ \mathfrak{m}\subseteq (b_0R+\mathfrak{m})_v\subseteq R $. By assumption, $\mathfrak{m}$ is a maximal $v$-ideal, hence $(b_0R+\mathfrak{m})_v= R.$ Since $\mathsf{T}(R)=\mathsf T(D),$ the quotient groups $\mathbf{q}(R^\bullet)=\mathbf{q}(D^\bullet)$ also coincide. 
 Then there exists $t\in R^\bullet$ such that $tc\epsilon_0\in R^\bullet.$ Observe also that $c\epsilon_0\mathfrak{m}\subseteq R$ implies  $tc\epsilon_0\mathfrak{m} \subseteq tR.$ Hence $$tbcR+ tc\epsilon_0\mathfrak{m}\subseteq tbcR+ tR=tR.$$ Therefore the following inclusions hold
\begin{equation*}
    \begin{split}
        tc\epsilon_0\in tc\epsilon_0R &=tc\epsilon_0(b_0R+ \mathfrak{m})_v\\
        &= (tc\epsilon_0(b_0R+  \mathfrak{m}))_v\\
        &=(tc\epsilon_0b_0R+ tc\epsilon_0\mathfrak{m})_v\\
        &=(tcbR+ tc\epsilon_0\mathfrak{m})_v \\ 
        &\subseteq (tR)_v =t(R_v)=tR.
    \end{split}
\end{equation*}
Since $t$ is regular, we get that $c\epsilon_0\in R.$ The statement follows by setting $\epsilon=\epsilon_0^{-1},\eta=\epsilon_0, v= b_0,w=c\epsilon_0.$
	\end{proof}

\section{Krull rings and C-rings}\label{4}
Krull rings with zero-divisors were introduced in \cite{Hu76} and then further developed independently by  Kennedy \cite{Ken80} and by Portelli and Spangher \cite{PS83}. Integrally closed noetherian rings are Krull rings. C-rings were introduced in \cite{GRR15} to study the arithmetic of non-integrally closed higher dimensional noetherian rings, generalizing the concepts of C-monoids and C-domains. Krull rings with finite class group are C-rings (see \cite[Corollary 4.4]{GRR15}) and further examples are given in Proposition \ref{prop: crings}. For an historical overview we refer to \cite{CO23}.

\smallskip

\subsection{Krull rings} Let $T$ be a commutative ring.  A \defit{rank-one discrete valuation} on $T$ is a surjective map $\vv\colon T \to \Z\cup \{\infty\}$ such that for all $x,y\in T$ \begin{enumerate}[label=\textup(\roman*\textup)]
\item $\vv(xy)=\vv(x)+\vv(y)$,
\item $\vv(x+y)\ge \min\{\vv(x),\vv(y)\},$
\item $\vv(1)=0$ and $\vv(0)=\infty. $
\end{enumerate} Let $R$ be a commutative ring. If there is a rank-one discrete valuation $\vv$ on $\mathsf{T}(R)$ such that \[R=\{x\in \mathsf{T}(R)\mid \vv(x)\ge 0\}\quad \text{and}\quad P=\{x\in \mathsf{T}(R)\mid \vv(x)>0\},\] then $(R,P)$ is called a \defit{rank-one discrete valuation pair} of $\mathsf{T}(R),$ and $R$ is called a \defit{rank-one discrete valuation ring} (shortly, rank-one DVR).
\begin{theorem}[{\cite[Theorem 3.5]{CK21}}]
    The following two statements are equivalent for a ring $R$.
    \begin{enumerate}[label=\textup(\normalfont\alph*\textup)]
        \item There exists a family $\{(V_\alpha,P_\alpha)\mid \alpha \in \Delta\}$ of rank-one discrete valutation pairs of $\mathsf{T}(R)$ with associated valuations $\{\vv_\alpha\mid \alpha\in \Delta\}$ such that 
\begin{enumerate}[label=\textup(\normalfont\roman*\textup)]
    \item $R=\bigcap_{\alpha\in \Delta}V_\alpha,$
    \item $P_\alpha$ is a regular ideal for all $\alpha\in \Delta,$
    \item for each regular $x\in \mathsf{T}(R)$, $\vv_\alpha(x)=0$ for all but finitely many $\alpha\in \Delta.$
\end{enumerate}
\item $R$ is a completely integrally closed Mori ring.
    \end{enumerate} If one of the equivalent conditions above holds, we say that $R$ is a \textsf{Krull ring}.
\end{theorem} For further characterizations of Krull rings see again \cite[Theorem 3.5]{CK21}. Moreover, we say that a monoid is \defit{Krull} if it is completely integrally closed and Mori.

\begin{proposition}[{\cite[Theorem 3.5]{GRR15}}] Let $R$ be a $v$-Marot ring. Then $R$ is a Mori \textup(resp. Krull\textup) ring if and only if $R^\bullet$ is a Mori \textup(resp. Krull\textup) monoid.
\end{proposition} 

\begin{remark}
    A normal ring is a ring $R$ that is reduced (i.e. it has no nilpotent elements) and integrally closed in his total quotient ring (see for instance \cite[Chapter 11.2]{Eis95}). In particular, by \cite[Theorem 10.1]{Hu88}, normal noetherian rings are examples of Krull rings. These rings play a significant role in algebraic geometry, where their importance is highlighted by Serre’s Normality Criterion \cite[Theorem 11.5]{Eis95}, which provides a complete characterization of normal noetherian rings. Furthermore, any non-zero normal noetherian ring is isomorphic to a finite direct product of normal domains. 
\end{remark}

\smallskip

\subsection{C-rings}  
    Let $F$ be a factorial monoid, and $H\subseteq F$ be a submonoid. Two elements $y,y'\in F$ are called $H$-\defit{equivalent} if for all $x\in F$, we have $xy\in H$ if and only if $xy'\in H.$ 
By definition, $H$-equivalence is a congruence relation on $F$, and for $y\in F,$ we denote by $[y]_H^F$ the $H$-\defit{equivalence class} of $y.$ 
 Then \[\mathcal{C}^*(H,F)=\left\{[y]_H^F\mid y\in (F\setminus F^\times)\cup \{1\}\right\}\] is a commutative semigroup, called the \defit{reduced class semigroup} of $H$ in $F$.

\begin{definition} Let $H$ be a monoid, and let $R$ be a ring.
\begin{enumerate}
    \item  $H$ is called a C-\defit{monoid} if it is a submonoid of a factorial monoid $F$ such that $H\cap F^\times=H^\times$ and $\mathcal{C}^*(H,F)$ is finite;
    \item $R$ is called a C-\defit{ring} if $R^{\bullet}$ is a C-monoid.
\end{enumerate}
\end{definition}
Let $H$ be a C-monoid. Then $H$ is Mori, the conductor $(H\DP \widehat{H})\ne \emptyset,$ and $\mathcal{C}_v(\widehat{H})$ is finite \cite[Chapter 10]{GHK06}. For arithmetic results and open problems on C-monoids, we refer to \cite{FG05, GHK06, GKL26, Rei13}.

\smallskip

\subsection{Transfer Krull rings and transfer C-rings} A monoid $H$ is said to be a \defit{transfer Krull monoid} if there exists a Krull monoid $B$ and a transfer homomorphism $\theta \colon H \to B$. Since the identity homomorphism is trivially a transfer homomorphism, Krull monoids are transfer Krull, but transfer Krull monoids need be neither Mori nor completely integrally closed. Another important class of transfer Krull monoids is half-factorial monoids. In that case, for a half-factorial monoid $H$, the map $\theta : H \to (\mathbb N_0, +)$, $a \mapsto \max \mathsf L(a)$, is a transfer homomorphism. But in general,
transfer Krull monoids are neither Krull monoids nor half-factorial monoids. An interested reader could see \cite{BR22}. Finally, a \defit{transfer Krull ring} is a ring $R$ such that $R^\bullet$ is a transfer Krull monoid.
\begin{theorem}\label{thm: transferkrull}
    Let $D$ be a ring and $R\subseteq D$ a subring with $\mathsf{T}(R)=\mathsf{T}(D)$ such that \[D = RD^{\times}, \  D^{\times} \cap R = R^{\times}, \quad \text{and} \quad { (R \DP D) \in v\text{-}\max (R)}  \,.\]
    \begin{enumerate}[label=\textup(\normalfont\arabic*\textup)]
        \item If $D$ is a Krull ring, then $D^\bullet$ is a Krull monoid, $R^\bullet \hookrightarrow D^\bullet$ is a transfer homomorphism, and $\mathcal{L}(R)=\mathcal{L}(D).$ 
        \item If $D$ is a $\mathrm{C}$-ring, then $D^\bullet$ is a $\mathrm{C}$-monoid, $R^\bullet \hookrightarrow D^\bullet$ is a transfer homomorphism, and $\mathcal{L}(R)=\mathcal{L}(D).$
    \end{enumerate}
\end{theorem}

\begin{proof}
    (1) The monoid $D^\bullet$ is a Krull monoid by \cite[Theorem 3.5(4)]{GRR15}. In addition, Theorem \ref{prop2} implies that $R^\bullet \hookrightarrow D^\bullet$ is a transfer homomorphism. Krull rings are Mori rings, hence $D^\bullet$ is a Mori monoid, and in particular atomic. Then Proposition \ref{proptransfhomo} implies that $\mathcal{L}(R)=\mathcal{L}(D).$ 

    (2) The monoid $D^\bullet$ is a C-monoid by definition, and it is a Mori monoid by \cite[Proposition 4.3]{GRR15}. Then $R^\bullet \hookrightarrow D^\bullet$ is a transfer homomorphism and $\mathcal{L}(R)=\mathcal{L}(D)$ by Theorem \ref{prop2}.
\end{proof}

The significance of the previous theorem is that it transfers the properties of the Krull ring (resp. C-ring) $D$ to the monoid  of regular elements of its subring $R.$
See \cite{Sch16} for a recent survery on set of lengths of Krull monoids, and \cite{Oh22, Rei13} for properties on C-monoids.

\smallskip

Next, we will discuss some Krull rings and C-rings that satisfy the assumptions of Theorem \ref{prop2}.

\begin{proposition}\label{prop: crings}
    Let $R$ be a ring and $D:=\widehat{R}$ its complete integral closure.
    \begin{enumerate}[label=\textup(\normalfont\arabic*\textup)]
        \item In the following cases, $D$ is a Krull ring.
            \begin{enumerate}[label=\textup(\normalfont\roman*\textup)]
                \item $R$ is noetherian.
                \item $R$ is $v$-Marot Mori such that $(R\DP D)$ is regular.
                \item $R$ is seminormal, i.e. $R=\{x\in \mathsf{T}(R)\mid x^n\in R \text{ for all sufficiently large } n\in \N\}.$
            \end{enumerate} In all cases, we have $D^\times \cap R=R^\times.$ Thus, if in addition, $D=RD^\times$ and $(R\DP D)\in v$-$\max(R)$, then the assumptions of Theorem \textsc{\ref{thm: transferkrull}} hold.
        \item If $R$ is $v$-Marot Mori such that $(R \DP D)$ is regular and $D/(R \DP  D)$ and $\mathcal C (D)$ are finite, then $R$ is a $\mathrm{C}$-ring.
    \end{enumerate}
\end{proposition}
\begin{proof}
    (1) If $R$ is noetherian, the statement follows from \cite[Theorem 10.1]{Hu88}. If $R$ is $v$-Marot Mori such that $(R\DP \widehat{R})$ is regular, then $R$ is a Krull ring by \cite[Theorem 4.8]{GRR15}. Finally, if $R$ is seminormal, then $R$ is Krull by \cite[Theorem 5.2.8]{HK25}.

    (2) It follows from \cite[Corollary 4.9]{GRR15}.
\end{proof}

\subsection{The $D+M$ construction} We discuss further applications of Theorem \ref{prop2} to rings which need not to be Krull.
 A $D+M$ ring is the sum of a ring and a maximal ideal. Inside the class of domains, the most common examples are the domains \[K+XL[X],\quad \text{and} \quad K+XL\llbracket X \rrbracket,\quad \text{where} \quad K\subseteq L\quad \text{are fields.}\] 

In what follows we show that rings of the form $D+M$ where $D$ is a ring and $M$ is one if its maximal ideals are an examples for which Theorem \ref{prop2} applies.  This will also gives us an explicit example of a transfer Krull ring. 

\begin{proposition}\label{D+m prop}
		Let $D$ be a ring, $\{0\}\ne \mathfrak{m}\in \max(D)$ be a maximal ideal of $D$, and $L\subseteq D$ a subfield such that $D=L+\mathfrak{m}$. Let $K\subsetneq L$ be a subfield and $R=K+\mathfrak{m}$. Assume that $\mathsf{T}(R)=\mathsf{T}(D)$, then
        \begin{enumerate}[label=\textup(\normalfont\arabic*\textup)]
            \item $D=RD^{\times},\, D^{\times} \cap R= R^{\times},  \, \text{and}\,\, (R\DP D)=\mathfrak{m}\in \max(R).$
            \item The embedding $R^\bullet \hookrightarrow D^\bullet$ is a  transfer homomorphism.
        \end{enumerate}
	\end{proposition}	

    \begin{proof}
        (1) Clearly, $L^\times R\subseteq D^\times R\subseteq D$, and $R^\times \subseteq D^\times \cap R$. Moreover, $\mathfrak{m}$ is a common ideal of $R$ and $D$, and $(R\DP D)$ is by definition the largest common ideal of $R$ and $D$, hence we get that $\mathfrak{m}\subseteq (R\DP D)\subseteq R$. 
			
			Let $a\in D$, then $a=u+m$, for some $u\in L$ and $m\in \mathfrak{m}$. If $u=0$, then $a=m\in R\subseteq L^\times R$. If $u\ne 0$, then $u$ is an invertible element 
			and $u^{-1}a=1+u^{-1}m\in 1+\mathfrak{m}\subseteq R$, whence $a =u(u^{-1}a)\in L^\times R$. This implies that $D=L^\times R=D^\times R$.
			
			Let $x\in D^\times \cap R$ such that $x=u+m \in R$ and $x^{-1}=u'+m' \in D$ where $u\in K$, $u'\in L$, and $m,m'\in \mathfrak{m}$. Then $1=xx^{-1}=uu'+um'+u'm+mm'$, and $1-uu'\in L\cap \mathfrak{m}=\{0\}.$ 
			So this implies that $uu'=1$, and $u'=u^{-1}\in K$, thus $x^{-1}\in K+\mathfrak{m}=R$ whence $D^{\times} \cap R= R^{\times}$.
			
			If $0\ne z\in (R\DP D)$, and $x\in L\setminus K$, then $zx\in R\setminus K\subseteq \mathfrak{m}$ and $x\notin \mathfrak{m}$, hence $z\in \mathfrak{m}$. 

            (2) This follows directly from part (1) and Theorem \ref{prop2}.
    \end{proof}

\begin{example}
Let $L/K$ be a  field extension, let $L[X]$ be the polynomial ring over $L$, let
$n\ge 2$ be an integer, $L[\theta]=L[X]/(X^n)$, and $R_n=K+\theta L[\theta]$. So $$R_n=\{a_0+a_1\theta+\cdots+a_{n-1}\theta^{n-1}\mid a_0\in K,a_i\in L \text{~for all~} i \in [1, n-1]\}.$$ By \cite[Proposition 2.1]{CK23}, $\mathsf{T}(R_n)=L[\theta]$, and $Z(R_n)=\theta L[\theta].$
Hence the maximal ideal $\mathfrak{m}=\theta L[\theta]$ of $L[\theta]$ is not regular, but the assumptions of Proposition \ref{D+m prop} are still satisfied. In particular, $R_n^\bullet \hookrightarrow L[\theta]^\bullet $ is a transfer homomorphism, and $R_n$ is a transfer Krull ring (indeed, since $L[\theta]$ is a total quotient ring, so in particular a Krull ring).  Notice that the ring $R_n$ can be seen as a pullback from the following: 
\[\begin{tikzcd}
	R_n\ar[rr]\ar[d,hook] & & K \ar[d,hook] \\ L[\theta] \ar[rr] & & L.
\end{tikzcd}\] 
\end{example}
 
For pullback constructions in rings with zero-divisors, see \cite{CF21, CK23, ZCKZ24}.

\section{Primary rings and weakly Krull rings}\label{5} Weakly Krull domains and weakly Krull monoids were introduced in the 1990s by Anderson, Mott, and Zafrullah \cite{AMZ92} and Halter-Koch \cite{HK98}, respectively. Weakly Krull rings were introduced only recently by Chang, and Oh \cite{CO22}. The arithmetic of weakly Krull Mori monoids $H$ found a lot of attention in factorization theory, in particular in the case when the conductor $(H\DP \widehat{H})\ne \emptyset$ (for a sample, see \cite{CFW22, FW22, GHK06, GKR15}). Thus by Theorem \ref{prop:height1}, if $R$ is a $v$-Marot weakly Krull Mori ring, then all these results hold true for $R^\bullet$ and $\mathcal{I}_v^*(R).$ A trivial example of weakly Krull rings are Marot primary rings.

\subsection{Primary rings}
\begin{definition}Let $H$ be a monoid, and let $R$ be a ring. 
    \begin{enumerate}
        \item $H$ is  \defit{primary} if $H\ne H^\times$, and for every $a,b\in H\setminus H^\times$ there exists $n\in \N$ such that $b^n\subseteq aH;$
        \item $H$ is \defit{strongly primary} if $H\ne H^\times$, and for every $a\in H\setminus H^\times$ there exists $n\in \N$ such that $(H\setminus H^\times)^n\subseteq aH;$
        \item $R$ is  \defit{primary} (resp. \defit{strongly primary}) if the monoid $R^\bullet$  of its regular elements is primary (resp. strongly primary).
    \end{enumerate}
\end{definition}
\begin{lemma}[{\cite[Lemma 2.7.7]{GHK06}}]\label{lemma:primary}
    Let $H$ be a monoid such that $H\ne H^\times$. Then $H$ is primary if and only if s-$\spec(R)=\{\emptyset, H\setminus H^\times\}.$
\end{lemma}

\begin{lemma}\label{lemma:HK}
    Let $R$ be a ring, let $I$ be a $t$-ideal of $R$ and let $P$ be a minimal prime ideal over $I$. Then $P$ is a prime $t$-ideal of $R$.
\end{lemma}
\begin{proof}
    This follows from \cite[Proposition 6.6]{HK98}.
\end{proof}
\begin{proposition}\label{lemma:stronglyprimary}
    Let $R$ be a $t$-Marot ring such that $R^\bullet\ne R^\times.$
  Then $R$ is primary if and only if $|t$-$\spec_r( R)|=1$. In particular, if $R$ is a Marot primary ring, then  $t$-$\spec_r(R)=\spec_r(R)=\{(R\setminus R^\times)R\}.$
\end{proposition}
\begin{proof}
      Suppose that $R$ is primary, then the monoid $R^\bullet$ is primary by definition, and so Lemma \ref{lemma:primary} implies that $s$-$\spec(R^\bullet)=\{\emptyset, R^\bullet\setminus R^\times\}$. Notice that the set of $t$-$\spec_r(R)$ is non-empty, since we assumed that $R\ne \mathsf{T}(R)$ \cite[Lemma 3.2.8(2)]{El19}. Let $P\in t$-$\spec_r(R).$ Then  $P^\bullet=P\cap R^\bullet$ is a non-empty prime $t$-ideal by \cite[Theorem 2.7]{CO22}, whence $P^\bullet=R^\bullet\setminus R^\times.$ Since $R$ is a $t$-Marot ring, we get that $P=(P^\bullet R)_{t_R}=((R^\bullet\setminus R^\times)R)_{t_R}.$ Therefore, $t$-$\spec_r(R)=\{((R^\bullet\setminus R^\times)R)_{t_R}\}.$  Conversely, suppose that $|t$-$\spec_r(R)|=1$, say $t$-$\spec_r(R)=\{P\}.$ We first show that $P\supseteq R^\bullet\setminus R^\times$. Let $x\in R^\bullet \setminus R^\times$. Then $xR$ is a proper $t$-ideal of $R$. Let $L$ be a minimal prime ideal of $xR$. It follows from Lemma \ref{lemma:HK} that $L$ is a prime $t$-ideal of $R$. Since $L$ is regular, we infer that $x\in L=P$. So we obtain that $((R^\bullet\setminus R^\times)R)_{t_R}\subseteq P_{t_R}=P=(P^\bullet R)_{t_R}\subseteq ((R^\bullet\setminus R^\times)R)_{t_R},$ thus $P=((R^\bullet\setminus R^\times)R=_{t_R}$. In order to show that $R^\bullet\setminus R^\times$ is the only non-empty prime $s$-ideal of $R^\bullet$, we prove that $R^\bullet\setminus R^\times$ is the only non-empty prime $t$-ideal of $R^\bullet$, since every prime $s$-ideal is the union of all prime $t$-ideals of $R^\bullet$ contained in it. Let $\mathfrak{q}$ be a non-empty prime $t$-ideal of $R^\bullet,$ and set $Q=(\mathfrak{q}R)_{t_R}.$ Then \cite[Theorem 2.7]{CO22} implies that $Q$ is a regular prime $t$-ideal of $R$ such that $Q\cap R^\bullet=\mathfrak{q}$. Since $Q=P=((R^\bullet\setminus R^\times)R)_{t_R},$ it follows that \[\mathfrak{q}=Q\cap R^\bullet = ((R^\bullet\setminus R^\times)R)_{t_R} \cap R^\bullet = R^\bullet \setminus R^\times.\] Consequently, $t$-$\spec(R^\bullet)=\{\emptyset, R^\bullet\setminus R^\times\},$ and therefore, $R^\bullet$ is primary. So $R$ is primary. 
      
     Moreover, if $R$ is a Marot primary ring and $P\in \spec_r(R)$, then $P=P^\bullet R=(R^\bullet\setminus R^\times)R\subseteq ((R^\bullet\setminus R^\times)R)_{t_R}\subsetneq R$
     since $R^\bullet$ is primary (see \cite[Theorem 2.7]{CO22}). 
     Then, as $R$ is Marot and $R^{\bullet} \setminus R^{\times}$ is a maximal  $s$-ideal of $R^{\bullet}$, we get that $((R^{\bullet} \setminus R^{\times})R)_{t_R}= (R^{\bullet} \setminus R^{\times})R$. In particular,
      $P=((R^\bullet\setminus R^\times)R)_{t_R}=P_{t_R}\in t$-$\spec_r(R).$
\end{proof}

\begin{remark}\label{rmk: falseprop}
    The previous proposition does not hold if we work with $v$-ideals instead of $t$-ideals. Indeed, the existence of maximal $v$-ideals is not always guaranteed, whereas maximal 
$t$-ideals always exist. As a counterexample, let $V$ be a non-discrete rank one valuation domain, and let $M$ be its maximal ideal. Then $V$ is primary, and $M_v=V$. In particular, $v$-$\spec(V)=\emptyset$.
\end{remark}
We recall the concept of locally tameness. Let $H$ be a monoid and let $u$ be an atom of $H$. If $u$ is prime, then we set $\mathbf{t}(H,u)=0.$ If $u$ is not prime, then $\mathbf{t}(H,u)$ is the smallest $N\in \mathbb N_0\cup \{\infty\}$ such that if $m\in \mathbb N$ and $v_1,\ldots, v_m$  are atoms of $H$ with $u\mid v_1\cdots v_m$, then there exists a subproduct of $v_1\cdots v_m$ which is a multiple of $u$, say $v_{i_1}\cdots v_{i_t}$, and a refactorization of this subproduct which contains $u$, say $v_{i_1}\cdots v_{i_t}=uu_2\cdots u_l$, such that $\max\{i_t,l\}\le N.$ We call $H$ \defit{locally tame} if all local tame degrees $\mathbf{t}(H,u)<\infty$ for all atoms $u\in H.$

\begin{proposition}\label{lemma:stronglyprimary1}
    Let $R$ be a $v$-Marot Mori ring such that $R^\bullet\ne R^\times$ and $|v$-$\spec_r(R)|=1$. Then $R$ is a strongly primary ring. Moreover, if $(R\DP \widehat{R})\ne \{0\},$ then $R^\bullet$ is locally tame. 
    
    In particular, a noetherian ring with exactly one regular prime ideal is strongly primary.
\end{proposition}
\begin{proof}
    The monoid $R^\bullet$ is Mori by assumption, hence the unique maximal s-ideal $\mathfrak{m}:=R^\bullet\setminus R^\times$ of $R^\bullet$ is minimal, and hence divisorial. Therefore there exists some finite subset $E\subseteq \mathfrak{m}$ such that $\mathfrak{m}=E_{v_{R^\bullet}}$ \cite[Proposition 2.1.10]{GHK06}. If $a\in \mathfrak{m},$ then, since $R$ is primary, there exists some $k\in \N$ such that $\{e^k\mid e\in E\} \subseteq aR^\bullet$. Hence $E^{k|E|}\subseteq aR^\bullet$, and $\mathfrak{m}^{k|E|}=(E_{v_{R^\bullet}})^{k|E|}\subseteq (E^{k|E|})_{v_{R^\bullet}}\subseteq aR^\bullet.$ 

     Moreover, if $(R\DP \widehat{R})\ne \{0\},$ then $(R^\bullet\DP \widehat{R^\bullet})\ne \emptyset,$ and $\widehat{R^\bullet}$ is Krull (see \cite[Theorem 2.3.5]{GHK06}) . Thus, $R^\bullet$ is locally tame by \cite[Theorem 3.5]{GHL07}.
\end{proof}
Proposition \ref{lemma:stronglyprimary1} generalizes the well-known result that one-dimensional local Mori domains are strongly primary. These domains are also locally tame, but strongly primary Mori monoids need not be locally tame in general. Much is known about the arithmetic of strongly primary monoids, in particular in case when they are locally tame (see for instance \cite[Theorem 4.1]{GGT21}). We post the following open problem.

\begin{problem}
    Let $R$ be a $v$-Marot Mori ring such that $|v$-$\spec_r(R)|=1.$ Is $R^\bullet$ locally tame?
\end{problem}
\begin{example}
    Let $k$ be a field, and $R:=k\llbracket X,Y\rrbracket/(X^2)$. Then $R$ is noetherian, with the unique regular prime ideal given by the ideal $(X,Y)$ generated by $X$ and $Y$. The previous lemma implies then that $R$ is strongly primary. 
\end{example}
\begin{definition} Let $H$ be a monoid, and let $R$ be a ring.
    \begin{enumerate}
        \item A non-unit $q$ of $H$ (resp. a non-zero non-unit $q$ of $R$) is called \defit{primary} if for all $b, c \in H$ (resp. for all $b,c\in R$) with $q \mid bc$ and $q \nmid b$, there is some $n \in \mathbb{N}$ such that $q \mid c^n$.
        \item The monoid $H$ (resp. the ring $R$) is said to be \defit{weakly factorial} if every non-unit of $H$ (resp. if every regular non-unit of $R$) is a finite product of primary elements of $H$ (resp. $R$).
    \end{enumerate}  We denote by $\mathcal{P}(H)$ the set of all primary elements of $H$, and by $\mathcal{P}(R)$ the set of all primary elements of $R$.
\end{definition}

A monoid $H$ is \defit{weakly Krull} if $H=\bigcap_{\mathfrak{p}\in \mathfrak{X}(H)}H_\mathfrak{p}$ and each element of $H$ is contained in only finitely many height-one primes of $H$. Here $\mathfrak{X}(H)$ denotes the set of height-1 prime ideal of $H.$  A ring $R$ is a \defit{weakly Krull ring} if $R=\cap_{P\in \mathfrak X_r(R)}R_{[P]}$, and each regular element of $R$ is contained in only finitely many regular height-1 prime ideals of $R.$ In particular, every Krull ring is a weakly Krull ring.

\begin{proposition}[\cite{CO22}]\label{prop: weaklykrull}
Let $R$ be a $t$-Marot ring. 
\begin{enumerate}
    \item $\mathcal{P}(R)\cap R^\bullet =\mathcal{P}(R^{\bullet})$. 
    \item $R$ is a weakly factorial ring \textup(resp. weakly Krull ring\textup) if and only if $R^{\bullet}$ is a weakly factorial monoid \textup(resp. weakly Krull monoid\textup).
    \item $R$ is a weakly factorial ring if and only if $R$ is a weakly Krull ring and $\mathcal{C}_t(R)=0.$
\end{enumerate}
\end{proposition}

\begin{proposition}
    Let $R$ be a ring. 
    \begin{enumerate}[label=\textup(\normalfont\arabic*\textup)]
        \item Suppose that $R$ is $v$-Marot and Mori. Then $R$ is weakly Krull if and only if $v$-$\spec(R)=\mathfrak{X}_r(R).$
        \item Suppose that $R$ is strongly primary. Then $R$ a transfer Krull if and only if $R$ is half-factorial.
    \end{enumerate} 
\end{proposition}
\begin{proof} (1)  The ring $R$ is $v$-Marot and Mori, therefore $R$ is $t$-Marot. Then by Proposition \ref{prop: weaklykrull}, $R$ is a weakly Krull ring if and only if $R^\bullet$ is a weakly Krull monoid. Moreover, by \cite[Theorem 24.5]{HK98}, $R^\bullet$ is a weakly Krull monoid if and only if $v$-$\spec(R^\bullet)=\mathfrak{X}(R^\bullet),$ whence the statement then follows from Proposition \ref{class group thm}.

    (2) The monoid $R^{\bullet}$ is a strongly primary monoid. By \cite[Theorem 5.5]{GSZ17} we get that $R^\bullet$ is a transfer Krull monoid if and only if $R^\bullet$ is a half-factorial monoid, therefore $R$ is a transfer Krull ring if and only $R$ is a half-factorial ring.
\end{proof}

\begin{remarks}\label{remarks}
\begin{enumerate}
    \item  For every ring $R$, we have that $R=\bigcap_{M\in \max_r(R)}R_{[M]}$ (see for instance \cite[Theorem 6.1]{Hu88}). Assume now that $R$ is a Marot primary ring, then $M:=(R^\bullet\setminus R^\times)R$ is the unique regular prime ideal. Hence $R=R_{[M]}=\bigcap_{P\in \mathfrak{X}_r(R)}R_{[P]}.$ Clearly each regular non-unit is contained in $M,$ whence $R$ is weakly Krull. 
    \item Let $R$ be a $v$-Marot ring and $P\in \mathfrak{X}_r(R).$ It is easy to verify that $[P]R_{[P]}$ is the only regular prime ideal of $R_{[P]}$. Therefore $R_{[P]}$ is primary by Proposition \ref{lemma:stronglyprimary}.
    \item Let $R$ be a ring, $H=R^\bullet$ the monoid of its regular elements, and $P\in \mathfrak{X}_r(P).$ Then $(R_{[P]})^\bullet= (H)_{P^\bullet}.$ Indeed, if $x\in (R_{[P]})^\bullet$, then there exists $s\in (R\setminus P)^\bullet$ such that $xs\in R.$ Thus $s\in H \setminus P,$ and $xs\in H.$ In particular, $x\in H_{P^\bullet}.$
\end{enumerate}
\end{remarks}

\begin{theorem}\label{prop:height1}
    Let $R$ be a v-Marot weakly Krull Mori ring, and let $H=R^\bullet$ be the monoid of its regular elements. Then \begin{enumerate}[label=\textup(\normalfont\arabic*\textup)]
        \item The map \[\Delta\colon\mathcal{I}_v^*(R)\to \coprod_{P\in \mathfrak{X}_r(R)}(R_{[P]})^\bullet_{\mathrm{red}} \] defined by $\Delta(I)(P)=x_PR_{[P]}^\times$ if $(I^\bullet)_{P^\bullet}=x_PH_{P^\bullet}$ for some $x_P\in H_{P^\bullet}$, is a monoid isomorphism. Moreover, $\mathcal{I}^*_v(R)$ is transfer Krull if and only if it is half-factorial.
        \item The monoids $R^\bullet$ and $\mathcal{I}^*_v(R)$ are weakly Krull. If $(R\DP \widehat{R})\ne\{0\},$ then $(H\DP \widehat{H})$ and $(\mathcal{I}_v^*(R)\DP \widehat{\mathcal{I}_v^*(R)})$ are non-empty.
        \item If $R$ is weakly factorial, then $R$ is a transfer Krull ring if and only if $R$ is half-factorial.
    \end{enumerate}
\end{theorem}
\begin{proof} 
    (1) By \cite[Proposition 5.3]{GKR15}, there is an isomorphism $\delta\colon\mathcal{I}_v^*(H)\to \coprod_{\mathfrak{p}\in \mathfrak{X}(H)}(H_{\mathfrak{p}})_{\mathrm{red}}$ that maps $\mathfrak{a}\in \mathcal{I}_v^*(H)$ into $\delta(\mathfrak{a})(\mathfrak{p})=a_\mathfrak{p}H_{\mathfrak{p}}^\times$ if $\mathfrak{a}_{\mathfrak{p}}=a_\mathfrak{p}H_{\mathfrak{p}}$ for some $a_\mathfrak{p}\in H_\mathfrak{p}$. Notice that the sets $\mathfrak{X}(H)$ and $\mathfrak{X}_r(R)$ are in bijection by Proposition \ref{class group thm}, and by Remarks \ref{remarks} we have that $(R_{[P]})^\bullet=(H)_{P^\bullet}$ for all $P\in \mathfrak{X}_r(R).$ 
    Consider the map  $\Delta \colon\mathcal{I}_v^*(R)\to \coprod_{P\in \mathfrak{X}_r(R)}(R_{[P]})^\bullet_{\mathrm{red}}$, that sends an invertible divisorial ideal $I$ into $\delta(I^\bullet)$. Hence $\Delta$ is an isomorphism since it is the composition of $\delta\colon \mathcal{I}_v^*(H)\to \coprod_{\mathfrak{p}\in \mathfrak{X}(H)}(H_{\mathfrak{p}})_{\mathrm{red}} $ and $\iota\colon \mathcal{I}_v^*(R)\to \mathcal{I}_v^*(H)$, and $\Delta(I)(P)=x_{P}H^\times_{P^\bullet}=x_{P}R^\times_{[P]}$ if $(I^\bullet)_{P^\bullet}=x_PH_{P^\bullet}$ for some $x_P\in H_{P^\bullet}$.
    
   Moreover, the monoids $(R_{[P]})^\bullet_{\mathrm{red}}= (H_{P^\bullet})_{\mathrm{red}}$ are Mori monoids, since $R$ is Mori by assumption and trivially primary. Therefore, they are strongly primary, whence by \cite[Theorem 5.5]{GSZ17} the monoids $(R_{[P]})^\bullet_{\mathrm{red}}$ are transfer Krull if and only if $(R_{[P]})^\bullet_{\mathrm{red}}$ are half-factorial. In particular, we get that $\coprod_{P\in \mathfrak{X}_r(R)}(R_{[P]})^\bullet_{\mathrm{red}} $ is transfer Krull if and only if it is half-factorial.

    (2) The monoid $R^\bullet$ is weakly Krull by \cite[Theorem 4.4]{CO22}. By Part (1), the monoid $\mathcal{I}_v^*(R)$ is isomorphic to $\coprod_{\mathfrak{p}\in \mathfrak{X}(H)}(H_\mathfrak{p})_{\mathrm{red}}$, and the latter is weakly Krull.
    
    (3) Since $R$ is weakly factorial Mori and $v$-Marot, then $R$ is a weakly Krull ring with the trivial class group \cite[Proposition 4.7]{CO22}. Then by part (1), $R^\bullet \cong \coprod_{P\in \mathfrak{X}_r(R)}(R_{[P]})^\bullet_{\mathrm{red}}$, and hence $R^\bullet$ is transfer Krull if and only if $R^\bullet$ is half-factorial.
\end{proof}

\section*{Acknowledgement}
We would like to express our gratitude to the anonymous referee for their valuable comments and constructive suggestions, which greatly improved the clarity and quality of this paper. In particular, we are grateful for Remark \ref{rmk: falseprop}, which helped us to properly formulate Proposition \ref{lemma:stronglyprimary}.

We are also deeply thankful to Alfred Geroldinger and Andread Reinhart for their suggestions and many helpful discussions throughout the development of this work.

\providecommand{\bysame}{\leavevmode\hbox to3em{\hrulefill}\thinspace}
\providecommand{\MR}{\relax\ifhmode\unskip\space\fi MR }
\providecommand{\MRhref}[2]{%
  \href{http://www.ams.org/mathscinet-getitem?mr=#1}{#2}
}
\bibliographystyle{alpha}

\end{document}